\documentclass[10pt]{amsart}

\usepackage{comment}
\usepackage[noadjust]{cite}
\usepackage{hyperref}
\usepackage[capitalise,noabbrev]{cleveref}
\usepackage{amsmath}
\usepackage{amssymb}
\usepackage{enumitem}
\usepackage{kbordermatrix}
\usepackage{multicol}

\usepackage{mathabx}

\usepackage{todonotes}

\usepackage[labelformat=simple]{subcaption}
\usepackage{tikz}
\usepackage{tkz-graph}
\usepackage{array}
\usepackage{mathabx}

\makeatletter
\newcommand{\thickhline}{%
    \noalign {\ifnum 0=`}\fi \hrule height 1pt
    \futurelet \reserved@a \@xhline
}
\newcolumntype{"}{@{\hskip\tabcolsep\vrule width 1pt\hskip\tabcolsep}}
\makeatother

\tikzstyle{VertexStyle} = [shape = circle, draw, fill]
\tikzset{pre/.style={-}}    
\tikzstyle{every node}=[circle, inner sep=0pt, minimum width=4pt]
\usetikzlibrary{backgrounds}
\DeclareMathOperator{\Asc}{Asc}
\DeclareMathOperator{\Cr}{Cr}
\DeclareMathOperator{\si}{si}
\DeclareMathOperator{\co}{co}

\newcommand{\GF}{\mathrm{GF}}
\newcommand{\ba}{\backslash}
\newcommand{\AG}{\mathit{AG}}

\newcommand{\TQ}{\mathit{TQ}}
\newcommand{\spikey}{spikey $3$-separator}

\newtheorem{theorem}{Theorem}[section]
\newtheorem{lemma}[theorem]{Lemma}

\newtheorem{conjecture}[theorem]{Conjecture}
\newtheorem{corollary}[theorem]{Corollary}

\setenumerate{label=\rm(\roman*),midpenalty=2}

\let\originalleft\left
\let\originalright\right
\renewcommand{\left}{\mathopen{}\mathclose\bgroup\originalleft}
\renewcommand{\right}{\aftergroup\egroup\originalright}

\newcommand{\ignore}[1]{}
\newcommand{\Z}{\mathbb{Z}}
 
\begin{document}

\title[Excluded minors for matroids representable over partial fields]{Computing excluded minors for classes of matroids representable over partial fields}
\author{Nick Brettell} 
\address{School of Mathematics and Statistics\\
  Victoria University of Wellington\\
  Wellington\\
  New Zealand}
\email{nick.brettell@vuw.ac.nz}
\author{Rudi Pendavingh} 
\address{Department of Mathematics and Computer Science\\
  Eindhoven University of Technology\\
  Eindhoven\\
  The Netherlands}
\email{R.A.Pendavingh@tue.nl}
\thanks{Supported by the ERC consolidator grant 617951. %
The first author was also supported by a Rutherford Foundation postdoctoral fellowship and the New Zealand Marsden Fund.}

\maketitle

\begin{abstract}
    We describe an implementation of a computer search for the ``small'' excluded minors for a class of matroids representable over a partial field.
    Using these techniques, we enumerate the excluded minors on at most $15$ elements for both the class of dyadic matroids, and the class of $2$-regular matroids.
    We conjecture that there are no other excluded minors for the class of $2$-regular matroids; whereas, on the other hand, we show that there is a $16$-element excluded minor for the class of dyadic matroids.
\end{abstract}

\section{Introduction}

A minor-closed class of matroids can be characterised by its \emph{excluded minors}: the minor-minimal matroids that are not in the class.
Finding an excluded-minor characterisation for a class of matroids representable over a certain field or fields is an area of much interest to matroid theorists (see \cite{GGK2000,HMvZ2011} for recent examples).
A class of matroids representable over a set of fields can be characterised by representability over a structure known as a \emph{partial field}.
Two particular tantalising classes of matroids representable over a partial field, for which excluded-minor characterisations are not yet known, are dyadic matroids and $2$-regular matroids.
In this paper, we describe an implementation of a computer search for the ``small'' excluded minors for a class of matroids representable over a partial field.
This approach was used to enumerate, by computer, the excluded minors on at most $15$ elements for the class of dyadic matroids, and for the class of $2$-regular matroids.

Our first result from this computation is the following:
\begin{theorem}
  The excluded minors for dyadic matroids on at most $15$ elements are
  $U_{2,5}$, $U_{3,5}$, $F_7$, $F_7^*$,
  $\AG(2,3)\ba e$, $(\AG(2,3)\ba e)^*$, $(\AG(2,3)\ba e)^{\Delta Y}$, $T_8$, $N_{1}$, $N_2$, and $N_{3}$.
\end{theorem}

With the exception of $N_3$, these matroids were previously known \cite[Problem~14.7.11]
{oxley}.
However, even this list is incomplete: we also found a $16$-element excluded minor that we call $N_4$.
We describe $N_3$ and $N_4$ in \cref{secdyadic}.

Our second result is the following:
\begin{theorem}
  \label{thm2}
  The excluded minors for $2$-regular matroids on at most $15$ elements are
$U_{2,6}$, $U_{3,6}$, $U_{4,6}$, $P_6$,
$F_7$, $F_7^*$, $F_7^-$, $(F_7^-)^*$, $F_7^=$, $(F_7^=)^*$,
$\AG(2,3)\ba e$, $(\AG(2,3)\ba e)^*$, $(\AG(2,3)\ba e)^{\Delta Y}$, $P_8$, $P_8^-$, $P_8^=$, and $\TQ_8$.
\end{theorem}
\noindent
The matroids $P_8^-$ and $\TQ_8$ are described in \cref{sec2reg}, whereas the others will be well known to readers familiar with the excluded-minor characterisations for $\GF(4)$-representable matroids \cite{GGK2000} and near-regular matroids \cite{HMvZ2011} (see also \cite{oxley}).

In the original version of this paper, we conjectured that this is the complete list of excluded minors for this class.
In fact, in recent work (appearing while this paper was under review), Brettell, Oxley, Semple and Whittle~\cite{BOSW2023a,BOSW2023b} proved that an excluded minor for the class of $2$-regular matroids has at most 15 elements.
Combining this result with \cref{thm2}, one obtains an excluded-minor characterisation of the class of $2$-regular matroids, which is the culmination of a long research programme~\cite{bww3,BCOSW2018,CMWvZ2015,CCCMWvZ2013,CCMvZ2015}.

The structure of this paper is as follows.  In the next section, we review preliminaries. 
In \cref{preconfine}, we introduce confined partial-field representations and describe how a representation over a partial field can be encoded by a representation over a finite field, with particular subdeterminants.
In \cref{impl}, we describe the implementation of the computation.  Rather than presenting the code (which we intend to make freely available), we focus on describing the implementation details that enabled us to search up to matroids on 15 elements using computer resources that are (more or less) readily available.  
In \cref{secdyadic,sec2reg}, we present our results for dyadic matroids and $2$-regular matroids, respectively.

\section{Preliminaries}
\label{prelims}

\subsection{Partial fields}
\label{prepf}

A \textit{partial field} is a pair $(R, G)$, where $R$ is a commutative ring with unity, and $G$ is a subgroup of the group of units of $R$ such that $-1 \in G$.
Note that $(\mathbb{F}, \mathbb{F}^*)$ is a partial field for any field $\mathbb{F}$.
If $\mathbb{P}=(R,G)$ is a partial field, then 
we write $p\in \mathbb{P}$ when $p\in G\cup \{0\}$, and $P\subseteq \mathbb{P}$ when $P\subseteq G\cup \{0\}$.

For disjoint sets $X$ and $Y$, we refer to a matrix with rows labelled by elements of $X$ and columns labelled by elements of $Y$ as an \emph{$X \times Y$ matrix}.
Let $\mathbb{P}$ be a partial field, and let $A$ be an $X\times Y$ matrix with entries from $\mathbb{P}$. Then $A$ is a $\mathbb{P}$-\textit{matrix} if every subdeterminant of $A$ is contained in $\mathbb{P}$. If $X'\subseteq X$ and $Y'\subseteq Y$, then we write $A[X',Y']$ to denote the submatrix of $A$ with rows labelled by $X'$ and columns labelled by $Y'$.
 
\begin{lemma}[{\cite[Theorem 2.8]{PvZ2010b}}]
\label{pmatroid}
Let $\mathbb{P}$ be a partial field, and let $A$ be an $X\times Y$ $\mathbb{P}$-matrix, where $X$ and $Y$ are disjoint sets. Let
\begin{equation*}
\mathcal{B}=\{X\}\cup \{X\triangle Z : |X\cap Z|=|Y\cap Z|, \det(A[X\cap Z,Y\cap Z])\neq 0\}. 
\end{equation*}
 Then $\mathcal{B}$ is the family of bases of a matroid on $X\cup Y$.
\end{lemma}

For an $X\times Y$ $\mathbb{P}$-matrix $A$, we let $M[A]$ denote the matroid in \cref{pmatroid}, and say that $A$ is a \emph{$\mathbb{P}$-representation} of $M[A]$.
Note that this is sometimes known as a reduced $\mathbb{P}$-representation in the literature; here, all representations will be ``reduced'', so we simply refer to them as representations.
A matroid $M$ is $\mathbb{P}$-\textit{representable} if there exists some $\mathbb{P}$-matrix $A$ such that 
$M \cong M[A]$.
We refer to a matroid $M$ together with a $\mathbb{P}$-representation $A$ of $M$ as a \emph{$\mathbb{P}$-represented} matroid.

For partial fields $\mathbb{P}_1$ and $\mathbb{P}_2$, we say that a function
$\phi : \mathbb{P}_1 \rightarrow \mathbb{P}_2$ is a \emph{homomorphism} if
\begin{enumerate}
  \item $\phi(1) = 1$,
  \item $\phi(pq) = \phi(p)\phi(q)$ for all $p, q \in \mathbb{P}_1$, and
  \item $\phi(p) +\phi(q) = \phi(p +q)$ for all $p, q \in \mathbb{P}_1$ such that $p +q \in \mathbb{P}_1$.
\end{enumerate}
Let $\phi([a_{ij}])$ denote $[\phi(a_{ij})]$.
The existence of a 
homomorphism from $\mathbb{P}_1$ to $\mathbb{P}_2$ certifies that $\mathbb{P}_1$-representability implies $\mathbb{P}_2$-representability:

\begin{lemma}[{\cite[Corollary 2.9]{PvZ2010b}}]
  \label{homomorphisms}
  Let $\mathbb{P}_1$ and $\mathbb{P}_2$ be partial fields and let $\phi : \mathbb{P}_1 \rightarrow \mathbb{P}_2$ be a 
  homomorphism.
  If a matroid is $\mathbb{P}_1$-representable, then it is also $\mathbb{P}_2$-representable.
  In particular, 
  if $A$ is a $\mathbb{P}_1$-representation of a matroid $M$, then $\phi(A)$ is a $\mathbb{P}_2$-representation of $M$.
\end{lemma}

Representability over a partial field can be used to characterise representability over each field in a set of fields.  Indeed,
for any finite set of fields $\mathcal{F}$, there exists a partial field~$\mathbb{P}$ such that a matroid is 
$\mathcal{F}$-representable
if and only if it is $\mathbb{P}$-representable \cite[Corollary~2.20]{PvZ2010a}.

  Let $M$ be a matroid.
  Pendavingh and Van Zwam described~\cite[Section~4.2]{PvZ2010b} the canonical construction of a partial field $\mathbb{P}_M$ with the property that for every partial field $\mathbb{P}$, the matroid $M$ is $\mathbb{P}$-representable if and only if there exists a homomorphism $\phi : \mathbb{P}_M \rightarrow \mathbb{P}$ (see also \cite{BL21b}).
  We call the partial field $\mathbb{P}_M$ the \emph{universal partial field of $M$}.

Let $\mathbb{P}=(R,G)$ be a partial field. 
We say that $p \in \mathbb{P}$ is \emph{fundamental} if $1-p \in \mathbb{P}$.
We denote the set of fundamentals of $\mathbb{P}$ by $\mathfrak{F}(\mathbb{P})$.
For $p \in \mathbb{P}$, the set of \emph{associates} of $p$ is
$$\Asc(p) = \begin{cases}\left\{p, 1-p, \frac{1}{p}, \frac{1}{1-p}, \frac{p}{p-1}, \frac{p-1}{p}\right\} & \textrm{if $p \notin \{0,1\}$}\\ \{0,1\} & \textrm{if $p \in \{0,1\}$.}\end{cases}$$
For $P \subseteq \mathbb{P}$, we write $\Asc(P) = \bigcup_{p \in P}\Asc(p)$.
If $p \in \mathfrak{F}(\mathbb{P})$, then $\Asc(p) \subseteq \mathfrak{F}(\mathbb{P})$.

Let $A$ and $A'$ be $\mathbb{P}$-matrices.  We write $A \preceq A'$ if $A$ can be obtained from $A'$ by the following operations: multiplying a row or column by an element of $G$,
deleting a row or column,
permuting rows or columns, 
and pivoting on a non-zero entry.
The \emph{cross ratios} of $A$ are $$\Cr(A) = \left\{p : \begin{bmatrix}1 & 1 \\ p & 1\end{bmatrix} \preceq A \right\}.$$

Any other undefined terminology related to partial fields follows Pendavingh and Van Zwam \cite{PvZ2010a,PvZ2010b}.
We note that although we work only at the generality of partial fields, this theory has been generalised by Baker and Lorscheid~\cite{BL20,BL21}.

\subsection{Partial fields of note}

The \emph{dyadic} partial field is
$\mathbb{D} = \left(\mathbb{Z}\left[\frac{1}{2}\right], \left<-1,2\right>\right)$.
We say a matroid is \emph{dyadic} if it is $\mathbb{D}$-representable.
A matroid is dyadic if and only if it is both $\GF(3)$-representable and $\GF(5)$-representable.
Moreover, a dyadic matroid is representable over every field of characteristic not two \cite[Lemma~2.5.5]{vanZwam2009}.

The \emph{2-regular} partial field is
$$\mathbb{U}_2 = (\mathbb{Q}(\alpha, \beta),\left<-1,\alpha, \beta, 1-\alpha, 1-\beta,\alpha-\beta\right>),$$
where $\alpha$ and $\beta$ are indeterminates.
We say a matroid is \emph{2-regular} if it is $\mathbb{U}_2$-representable.
Note that $\mathbb{U}_2$ is the universal partial field of $U_{2,5}$ \cite[Theorem 3.3.24]{vanZwam2009}.
If a matroid is $2$-regular, then it is $\mathbb{F}$-representable for every field $\mathbb{F}$ of size at least four \cite[Corollary 3.1.3]{Semple1998}.
However, the converse does not hold; for example, $U_{3,6}$ is representable over all fields of size at least four, but is not $2$-regular~\cite[Lemma 4.2.4]{Semple1998}.

More generally, the \emph{$k$-regular} partial field is $$\mathbb{U}_k = (\mathbb{Q}(\alpha_1,\dots,\alpha_k), \left<\{x-y : x,y \in \{0,1,\alpha_1,\dotsc,\alpha_k\}\textrm{ and }x \neq y\}\right>),$$ where $\alpha_1,\dotsc,\alpha_k$ are indeterminates.
In particular, a matroid is \emph{near-regular} if it is $\mathbb{U}_1$-representable.

We also make some use of the following partial fields \cite{PvZ2010b,vanZwam2009}.
The \emph{sixth-root-of-unity} partial field is
$\mathbb{S} = \left(\mathbb{Z}\left[\zeta\right], \left<\zeta\right> \right)$, where $\zeta$ is a solution to $x^2 - x + 1 = 0$.
A matroid is $\mathbb{S}$-representable if and only if it is $\GF(3)$- and $\GF(4)$-representable.

The \emph{2-cyclotomic} partial field is
$$\mathbb{K}_2 = (\mathbb{Q}(\alpha),\left<-1,\alpha-1, \alpha, \alpha+1\right>),$$
where $\alpha$ is an indeterminate.
If a matroid is $\mathbb{K}_2$-representable, then it is representable over every field of size at least four; but the converse does not hold \cite[Lemma~4.14 and Section~6]{PvZ2010a}.
The class of $2$-regular matroids is a proper subset of the $\mathbb{K}_2$-representable matroids.

\sloppy
Finally, Pendavingh and Van Zwam introduced, for each $i \in \{1,\dotsc,6\}$, the \emph{Hydra-$i$} partial field $\mathbb{H}_i$~\cite{PvZ2010b}.
A $3$-connected quinary matroid with a $\{U_{2,5},U_{3,5}\}$-minor is $\mathbb{H}_i$-representable if and only if it has at least $i$ inequivalent $\GF(5)$-representations.

\subsection{Delta-wye exchange}
\label{predy}

Let $M$ be a matroid with a coindependent triangle $T=\{a,b,c\}$.
Consider a copy of $M(K_4)$ having $T$ as a triangle with $\{a',b',c'\}$ as the complementary triad labelled such that $\{a,b',c'\}$, $\{a',b,c'\}$ and $\{a',b',c\}$ are triangles.
Let $P_{T}(M,M(K_4))$ denote the generalised parallel connection of $M$ with this copy of $M(K_4)$ along the triangle $T$.
Let $M'$ be the matroid $P_{T}(M,M(K_4))\backslash T$ where the elements $a'$, $b'$ and $c'$ are relabelled as $a$, $b$ and $c$ respectively.
The matroid~$M'$ is said to be obtained from $M$ by a \emph{$\Delta$-$Y$ exchange} on the triangle~$T$.
Dually, $M''$ is obtained from $M$ by a \emph{$Y$-$\Delta$ exchange} on the triad $T^*=\{a,b,c\}$ if $(M'')^*$ is obtained from $M^*$ by a $\Delta$-$Y$ exchange on $T^*$. 

We say that matroids $M$ and $M'$ are \emph{$\Delta Y$-equivalent} if $M'$ can be obtained from $M$ by a (possibly empty) sequence of $\Delta$-$Y$ exchanges on coindependent triangles and $Y$-$\Delta$ exchanges on independent triads.

For a matroid~$M$, we use $\Delta(M)$ to denote the set of all matroids $\Delta Y$-equivalent to $M$; for a set of matroids $\mathcal{N}$, we use $\Delta(\mathcal{N})$ to denote $\bigcup_{N \in \mathcal{N}} \Delta(N)$.
We also use $\Delta^{(*)}(\mathcal{N})$ to denote $\bigcup_{N \in \mathcal{N}} \Delta(\{N,N^*\})$.

The following two results were proved by Oxley, Semple and Vertigan~\cite{OSV2000}, generalising the analogous results by Akkari and Oxley~\cite{AO1993} regarding the $\mathbb{F}$-representability of $\Delta Y$-equivalent matroids for a field $\mathbb{F}$.

\begin{lemma}[{\cite[Lemma 3.7]{OSV2000}}]
  \label{deltaYrep}
  Let $\mathbb{P}$ be a partial field, and let $M$ and $M'$ be $\Delta Y$-equivalent matroids.
  Then $M$ is $\mathbb{P}$-representable if and only if $M'$ is $\mathbb{P}$-representable.
\end{lemma}

\begin{lemma}[{\cite[Theorem 1.1]{OSV2000}}]
  \label{deltaYexc}
  Let $\mathbb{P}$ be a partial field, and let $M$ be an excluded minor for the class of $\mathbb{P}$-representable matroids.
  If $M'$ is $\Delta Y$-equivalent to $M$, then $M'$ is an excluded minor for the class of $\mathbb{P}$-representable matroids.
\end{lemma}

\subsection{Excluded-minor characterisations}

We now recall Geelen, Gerards and Kapoor's excluded-minor characterisation of quaternary matroids~\cite{GGK2000}.
The matroid $P_8$ is illustrated in \cref{p8fig};
observe that $\{a,b,c,d\}$ and $\{e,f,g,h\}$ are disjoint circuit-hyperplanes.
Relaxing both of these circuit-hyperplanes results in the matroid $P_8^=$.

\begin{theorem}[{\cite[Theorem~1.1]{GGK2000}}]
  \label{gf4minors}
  A matroid is $\GF(4)$-representable if and only if it has no minor isomorphic to $U_{2,6}$, $U_{4,6}$, $P_6$, $F_7^-$, $(F_7^-)^*$, $P_8$, and $P_8^=$.
\end{theorem}

Let $\AG(2,3) \ba e$ denote the matroid obtained from $\AG(2,3)$ by deleting an element (this matroid is unique up to isomorphism).
Let $(\AG(2, 3)\ba e)^{\Delta Y}$ denote matroid obtained from $\AG(2,3) \ba e$ by performing a single $\Delta$-$Y$ exchange on a triangle (again, this matroid is unique up to isomorphism).
Hall, Mayhew, and Van Zwam proved the following excluded-minor characterisation of the near-regular matroids~\cite{HMvZ2011}.

\begin{theorem}[{\cite[Theorem 1.2]{HMvZ2011}}]
  \label{nr_exminors}
  A matroid is near-regular if and only if it has no minor isomorphic to $U_{2,5}$, $U_{3,5}$, $F_7$, $F_7^*$, $F_7^-$, $(F_7^-)^*$, $\AG(2,3)\ba e$, $(\AG(2,3)\ba e)^*$, $(\AG(2, 3)\ba e)^{\Delta Y}$, and $P_8$.
\end{theorem}

\subsection{Splitter theorems}

Let $\mathcal{N}$ be a set of matroids.
We say that a matroid $M$ has an $\mathcal{N}$-minor if $M$ has an $N$-minor for some $N \in \mathcal{N}$.
In order to exhaustively generate the matroids in some class that are $3$-connected and have an $\mathcal{N}$-minor,
we use Seymour's Splitter Theorem extensively.

\begin{theorem}[Seymour's Splitter Theorem \cite{Seymour1980}]
  \label{seysplit}
  Let $M$ be a $3$-connected matroid that is not a wheel or a whirl, and let $N$ be a $3$-connected proper minor of $M$.
  Then there exists an element $e \in E(M)$ such that $M/e$ or $M\ba e$ is $3$-connected and has an $N$-minor.
\end{theorem}

  We are primarily interested in matroids that are not near-regular, due to \cref{nr_exminors}.
  The next corollary follows from the observation that wheels and whirls are near-regular.
\begin{corollary}
  \label{seysplitcorr}
  Let $M$ be a $3$-connected matroid with a proper $N$-minor, where $N$ is not near-regular. 
  Then, for $(M',N') \in \{(M,N),(M^*,N^*)\}$, there exists an element $e \in E(M')$ such that $M'\ba e$ is $3$-connected and has an $N'$-minor.
\end{corollary}

To reduce the number of extensions to consider, when generating potential excluded minors, we use splicing, as described in \cref{sec-splic}.
Since we only keep track of $3$-connected matroids with a particular $N$-minor, we require a guarantee of the existence of so-called $N$-detachable pairs~\cite{bww3}, in order to generate an exhaustive list of potential excluded minors.
%
%
Let $M$ be a $3$-connected matroid, and let $N$ be a $3$-connected minor of $M$.
A pair $\{a,b\} \subseteq E(M)$ is \emph{$N$-detachable} if either $M\ba a\ba b$ or $M/a/b$ is $3$-connected and has an $N$-minor. 
To describe matroids with no $N$-detachable pairs, we require a definition.
Let $P \subseteq E(M)$ be an exactly $3$-separating set of $M$ such that $|P| \ge 6$.
Suppose $P$ has the following properties:
\begin{enumerate}[label=\rm(\alph*)]
  \item there is a partition $\{L_1,\dotsc,L_t\}$ of $P$ into pairs such that for all distinct $i,j\in\{1,\dotsc,t\}$, the set $L_i\cup L_j$ is a cocircuit,
  \item there is a partition $\{K_1,\dotsc,K_t\}$ of $P$ into pairs such that for all distinct $i,j\in\{1,\dotsc,t\}$, the set $K_i\cup K_j$ is a circuit,
  \item $M / p$ and $M \ba p$ are $3$-connected for each $p \in P$,
  \item for all distinct $i,j\in\{1,\dotsc,t\}$, the matroid $\si(M / a / b)$ is $3$-connected for any $a \in L_i$ and $b \in L_j$, and
  \item for all distinct $i,j\in\{1,\dotsc,t\}$, the matroid $\co(M \ba a \ba b)$ is $3$-connected for any $a \in K_i$ and $b \in K_j$.
\end{enumerate}
Then we say $P$ is a \emph{\spikey} of $M$.

\begin{theorem}[{\cite[Theorem 1.1]{bww3}}]
  \label{detachthm}
  Let $M$ be a $3$-connected matroid, and let $N$ be a $3$-connected minor of $M$ such that $|E(N)| \ge 4$, and $|E(M)|-|E(N)| \ge 6$.
  Then either
  \begin{enumerate}
    \item $M$ has an $N$-detachable pair, 
    \item there is a matroid $M'$ obtained by performing a single $\Delta$-$Y$ or $Y$-$\Delta$ exchange on $M$ such that $M'$ has an $N$-minor and an $N$-detachable pair, or
    \item $M$ has a \spikey\ $P$,
      and if $|E(M)| \ge 13$, then at most one element of $E(M)-E(N)$ is not in $P$. 
  \end{enumerate}
\end{theorem}

We note that in the statement of this theorem in \cite{bww3}, the precise structure of the $3$-separators that arise in case (iii) is described.
It is clear that when $|E(M)|-|E(N)| \ge 6$, each of these $3$-separators satisfy conditions (a) and (b) in the definition of a \spikey.
The fact that (c) holds for such a $3$-separator follows from \cite[Lemma~5.3]{bww3}, and it is easily checked that (d), and dually (e), also hold.

\subsection{Equivalence of \texorpdfstring{$\mathbb{P}$}{P}-matrices, and stabilizers}

Let $\mathbb{P} = (R,G)$ be a partial field, and let $A$ and $A'$ be $\mathbb{P}$-matrices.
We say that $A$ and $A'$ are \emph{scaling equivalent} if $A'$ can be obtained from $A$ by scaling rows and columns by elements of $G$.
If $A'$ can be obtained from $A$ by scaling, pivoting, permuting rows and columns, and also applying automorphisms of $\mathbb{P}$, then we say that $A$ and $A'$ are \emph{algebraically equivalent}.
We say that $M$ is \emph{uniquely representable over $\mathbb{P}$} if any two $\mathbb{P}$-representations of $M$ are algebraically equivalent.

Let $M$ and $N$ be $\mathbb{P}$-representable matroids, where $M$ has an $N$-minor.
Then \emph{$N$ stabilizes $M$ over $\mathbb{P}$} if for any scaling-equivalent $\mathbb{P}$-representations $A_1'$ and $A_2'$ of $N$ that extend to $\mathbb{P}$-representations $A_1$ and $A_2$ of $M$, respectively, $A_1$ and $A_2$ are scaling equivalent.

For a partial field~$\mathbb{P}$, let $\mathcal{M}(\mathbb{P})$ be the class of matroids representable over $\mathbb{P}$.
A matroid $N \in \mathcal{M}(\mathbb{P})$ is a \emph{$\mathbb{P}$-stabilizer} if, for any $3$-connected matroid $M \in \mathcal{M}(\mathbb{P})$ having an $N$-minor, the matroid $N$ stabilizes $M$ over $\mathbb{P}$.

Following Geelen et al.~\cite{GOVW1998}, we say that
a matroid $N$ \emph{strongly stabilizes $M$ over $\mathbb{P}$} if $N$ stabilizes $M$ over $\mathbb{P}$, and every
$\mathbb{P}$-representation of $N$ extends to a $\mathbb{P}$-representation of $M$.
We say that $N$ is a \emph{strong $\mathbb{P}$-stabilizer} 
if $N$ is a $\mathbb{P}$-stabilizer and $N$ strongly stabilizes every matroid in $\mathcal{M}(\mathbb{P})$ with an $N$-minor.

\section{Partial-field proxies} 
\label{preconfine}

In this section, we show that we can simulate a representation over a partial field by a representation over a finite field, where we have constraints on the subdeterminants appearing in the representation.
This has efficiency benefits for our computations, as we can utilise an existing implementation of finite fields, and avoid a full implementation of a partial field from scratch.

Let $\mathbb{P}$ be a partial field, let $F \subseteq \mathfrak{F}(\mathbb{P})$, let $M$ be a matroid, and
 let $A$ be a $\mathbb{P}$-matrix so that $M=M[A]$.
 We say that the matrix $A$ is  
 {\em $F$-confined} if $\Cr(A) \subseteq F \cup \{0,1\}$.
 If $A$ is an $F$-confined $\mathbb{P}$-matrix  and $\phi: \mathbb{P}\rightarrow \mathbb{P}'$ is a partial-field homomorphism, then $M[A]=M[\phi(A)]$ and $$\Cr(\phi(A)) \subseteq \phi(F),$$ so that $\phi(A)$ is an $\phi(F)$-confined representation over $\mathbb{P}'$.
We will show
 that under certain conditions on $\phi$ and $F$, any $\phi(F)$-confined representation over $\mathbb{P}'$ can be lifted to an $F$-confined representation over $\mathbb{P}$.

 The following is a reformulation of \cite[Corollary~3.8]{PvZ2010a} (see also \cite[Corollary~4.1.6]{vanZwam2009}) using the notion of $F$-confined partial-field representations.
To see this, take the restriction of $h$ to $\Cr(A)$ as the lift function.

\begin{theorem}[{Lift Theorem \cite{PvZ2010a}}]
  \label{lift}
  Let $\mathbb{P}$ and $\mathbb{P}'$ be partial fields, let $F \subseteq \mathfrak{F}(\mathbb{P}')$, let $A$ be an $F$-confined $\mathbb{P}'$-matrix, and let $\phi : \mathbb{P} \rightarrow \mathbb{P}'$ be a partial-field homomorphism.
  Suppose there exists a function $h: F \rightarrow \mathbb{P}$ such that
  \begin{enumerate}
    \item $\phi(h(p)) = p$ for all $p \in F$,
    \item if $1+1 \in \mathbb{P}'$, then $1+1 \in \mathbb{P}$, and $1+1 = 0$ in $\mathbb{P}'$ if and only if $1+1=0$ in $\mathbb{P}$,
    \item for all $p,q \in F$,
      \begin{itemize}
        \item if $p+q = 1$ then $h(p) + h(q) = 1$, and
        \item if $pq = 1$ then $h(p)h(q) = 1$; and,
      \end{itemize}
    \item for all $p,q,r \in F$, we have $pqr = 1$ if and only if $h(p)h(q)h(r) = 1$.
  \end{enumerate}
  Then there exists a $\mathbb{P}$-matrix $A'$ such that $\phi(A')$ is scaling-equivalent to $A$.
\end{theorem}

We are interested in the case where $\mathbb{P}'$ is a finite field $\mathbb{F}=\GF(q)$ for some prime power $q$. In this case, we obtain the following corollary:

\begin{corollary}
  \label{thm:proxy}
  Let $\mathbb{P}$ be a partial field, let $\mathbb{F}$ be a finite field,
  let $\phi : \mathbb{P} \rightarrow \mathbb{F}$ be a partial-field homomorphism,
  let $F = \phi(\mathfrak{F}(\mathbb{P}))$,
  and let $A$ be an $F$-confined $\mathbb{F}$-matrix.
  Suppose that the restriction of $\phi$ to $\mathfrak{F}(\mathbb{P})$ is injective,
  and 
  \begin{enumerate}
    \item for all $p, q\in \mathfrak{F}(\mathbb{P})$,  if $\phi(p)+\phi(q)=1$, then $p+q=1$; and
    \item for all $p,q,r\in \mathfrak{F}(\mathbb{P})$, if $\phi(p)\phi(q)\phi(r)=1$, then $pqr=1$; and 
    \item if $1=-1$ in $\mathbb{F}$, then $1=-1$ in $\mathbb{P}$.
  \end{enumerate}
  Then there exists a $\mathbb{P}$-matrix $A'$ such that $\phi(A')$ is scaling-equivalent to $A$.
\end{corollary}
\begin{proof}
  We work towards applying \cref{lift} with $\mathbb{P}' = \mathbb{F}$.
  Since the restriction of $\phi$ to $\mathfrak{F}(\mathbb{P})$ is injective and $\phi(\mathfrak{F}(\mathbb{P}))=F$, there is a well-defined function $h : F \rightarrow \mathfrak{F}(\mathbb{P})$ where $h(f) = p$ when $\phi(p) = f$.  Now $h$ is the inverse of $\phi|_{\mathfrak{F}(\mathbb{P})}$, and thus it is easily seen that (i)--(iv) of \cref{lift} are satisfied by the function $h$.
\end{proof}


\begin{corollary} \label{cor:dyad} $M$ is dyadic if and only if $M$ has a $\{2,6,10\}$-confined representation over $\GF(11)$.\end{corollary}
\proof Recall that $\mathfrak{F}(\mathbb{D}) \setminus \{0,1\} =\{-1,2, 2^{-1}\}$ \cite{vanZwam2009}.
Consider the partial-field homomorphism $d: \mathbb{D}\rightarrow \GF(11)$ defined by $d(2)=2$, $d(-1)=10$, $d(2^{-1})=6$. A finite check suffices to verify that the conditions of the theorem are satisfied for $(\mathbb{P}, \mathbb{F}, \phi) =  (\mathbb{D}, \GF(11), d)$, and that then $F = \{2,6,10\}$. The corollary follows.\endproof
A finite check reveals that we cannot take a smaller finite field $\mathbb{F}$ which admits a partial-field homomorphism $\phi: \mathbb{D}\rightarrow \mathbb{F}$ to take the role of $\GF(11)$ in this corollary.  For example, if we take $\mathbb{F}=\GF(7)$, then  $\phi(2)\phi(2)\phi(2)=1$, but $2\cdot 2\cdot 2\neq 1$.

Let $\mathbb{P}$ be a partial field.
For a finite field $\mathbb{F}$ and partial-field homomorphism $\phi : \mathbb{P} \rightarrow \mathbb{F}$, we say that $(\mathbb{F},\phi)$ is a \emph{proxy} for $\mathbb{P}$ if $\phi$ can be lifted in the sense of \cref{thm:proxy}.
For example, the proof of \cref{cor:dyad} shows that $(\GF(11),d)$ is a proxy for $\mathbb{D}$.

\Cref{fig:conff} lists several 
partial field proxies
(see \cite[Appendix~A]{PvZ2010b} for any partial fields undefined here).
These were found by an exhaustive search (by computer), trying each prime $p$, in order, until the desired homomorphism was found.
Note that, with the exception of $\mathbb{H}_4$ and $\mathbb{H}_5$, these are the smallest finite fields of prime order for which such a homomorphism exists (for these two partial fields, the search was time consuming, so we started it at a large prime).

\begin{table}[htb]
$\begin{array}{lll}
\text{Partial field }&\text{Finite Field }&\text{Partial field homomorphism}\\
\hline
\mathbb{S}& \GF(7) & \zeta\mapsto 3\\
 \mathbb{D}&\GF(11) & 2\mapsto 2\\
 \mathbb{G}&\GF(19) & \tau\mapsto 5\\
 \mathbb{U}_1& \GF(23) & \alpha\mapsto 5\\
 \mathbb{H}_2& \GF(29) & i\mapsto 12\\
 \mathbb{K}_2&\GF(73) & \alpha\mapsto 15\\
 \mathbb{H}_3& \GF(151) & \alpha\mapsto 4\\
 \mathbb{P}_4& \GF(197) & \alpha\mapsto 31\\
 \mathbb{U}_2&\GF(211) & \alpha\mapsto 4, \beta\mapsto 44\\
 \mathbb{H}_4& \GF(947) &\alpha\mapsto 272, \beta\mapsto 928\\
  \mathbb{H}_5& \GF(3527) &\alpha\mapsto 1249, \beta\mapsto 295, \gamma\mapsto 3517\\
\end{array}$
\caption{\label{fig:conff} Several proxies for partial fields.}
\end{table}

Each of the partial fields listed in Table \ref{fig:conff} has finitely many fundamentals. There necessarily exists a finite field proxy for such partial fields. To establish this, we will need the following fact.
\begin{lemma} \label{lem:fgfield} Let $R=\Z[X_1,\ldots, X_k]$, and let $J$ be a maximal ideal of $R$. Then $R/J$ is a finite field.
\end{lemma}

\proof  As $J$ is a maximal ideal of the ring $R$, $\mathbb{F}:=R/J$ is a field. 

Suppose that $\mathbb{F}$ is a field of characteristic $0$. Then the prime field $S$ of $\mathbb{F}$ is isomorphic to $\mathbb{Q}$.
$\mathbb{F}$ is finitely generated as an algebra over $\Z$, since  
 $$\mathbb{F}= \Z[X_1,\ldots, X_k]/J= \Z[a_1,\ldots, a_k]$$ where $a_i$ is the residue class of $X_i$ modulo $J$. Since $S\supseteq \Z$, $\mathbb{F}$ is also finitely generated as an algebra over the field $S$. By Zariski's Lemma \cite[Proposition 7.9]{AM1969}, it follows that $\mathbb{F}$ is a finite field extension of $S$.  So $\Z\subseteq S\subseteq \mathbb{F}$,  $\mathbb{F}$ is finitely generated as an algebra over $\Z$, and $\mathbb{F}$ is finitely generated as a module over $S$. By the Artin-Tate Lemma \cite[Proposition 7.8]{AM1969}, it then follows that $S\cong \mathbb{Q}$ is finitely generated as an algebra over $\Z$. Say, $\mathbb{Q}=\Z[t_1,\ldots, t_m]$ where $t_i=p_i/q_i$, with $p_i, q_i\in \Z$, and $q_i\neq 0$. 
Pick any prime $p$ that does not divide $q_i$ for any $i$. As $1/p\in \mathbb{Q}=\Z[t_1,\ldots, t_m]$, there is an integer polynomial $r\in \Z[X_1,\ldots, X_m]$ so that $1/p = r(t_1,\ldots, t_m)$. It follows that there exist integers $u,v\in \Z$ such that $1/p=u/v$ and $v$ is a power of $\prod_i q_i$. Then $v=up$, but $p$ does not divide $v$, a contradiction. 

So $\mathbb{F}$ is a field of characteristic $p>0$, that is, $p\in J$. Then 
$$\mathbb{F}=\Z[X_1,\ldots, X_k]/J=\GF(p)[X_1,\ldots, X_k]/J'=\GF(p)[b_1,\ldots, b_k],$$
 where $b_i$ is the residue class of $X_i$ modulo $J'$, and $J'\subseteq \GF(p)[X_1,\ldots, X_k]$ is $J$ modulo $p$.  So $\mathbb{F}$ is finitely generated as an algebra over $\GF(p)$. By Zariski's Lemma \cite[Proposition 7.9]{AM1969}, it follows that $\mathbb{F}$ is a finite field extension of $\GF(p)$. Then $\mathbb{F}=\GF(p^k)$ for some integer $k$, as required.
 \endproof
Lemma \ref{lem:fgfield} is perhaps not surprising to anyone familiar with the fundamentals of commutative algebra, but at the same time it is not elementary. We thank Rob Eggermont for providing us with a short proof (indeed, with three short proofs).

\begin{theorem} Let $\mathbb{P}$ be a partial field with  finitely many fundamentals. Then there exists a finite field $\mathbb{F}$ and homomorphism $\phi : \mathbb{P} \rightarrow \mathbb{F}$, so that  $(\mathbb{F},\phi)$ is a proxy for $\mathbb{P}$.
\end{theorem}
\proof Let $\mathbb{P}=(R,G)$ be a partial field such that $|\mathfrak{F}(\mathbb{P})|<\infty$. We may assume that $G$ is generated by $\mathfrak{F}(\mathbb{P})$ and that $R=\Z[G]$. 
Note that under these simplifying assumptions there is an ideal $I$ of $\Z[W]$, where $W:=\{W_f: f\in \mathfrak{F}(\mathbb{P})\}$, so that $R=\Z[W]/I$.

Consider the ring $S:=R[X, Y, Z]$ where $X, Y, Z$ are the collections of variables 
$$X:=\{X_{pq}: p, q\in \mathfrak{F}(\mathbb{P})\}\cup\{X_{11}\}, ~Y:=\{Y_{pq}: p, q\in \mathfrak{F}(\mathbb{P})\cup\{0\}, p+q\neq 1\}$$
and $Z:= \{Z_{pqr}: p, q, r\in \mathfrak{F}(\mathbb{P})\cup\{1\}, pqr\neq 1\}$.
Let $J'$ be the ideal of $S$ generated by
$$\{(p-q)X_{pq}-1: p, q\in \mathfrak{F}(\mathbb{P}), p\neq q\}$$
$$\{(p+q-1)Y_{pq}-1: \mathfrak{F}(\mathbb{P})\cup\{0\}, p+q\neq 1\}$$
  $$\{(pqr-1)Z_{pqr}-1: p, q,r\in \mathfrak{F}(\mathbb{P})\cup\{1\}, pqr\neq 1\}$$
 and the generator $2X_{11}-1$ if $1\neq -1$ in $\mathbb{P}$. Since each of the polynomials generating $J'$ uses a variable unique to that generator, the ideal $J'$ is proper, i.e. $1\not\in J'$.
 
Let $J$ be a maximal ideal of $S$ containing $J'$. As $S$ is commutative and $J$ is maximal,  $\mathbb{F}:=S/J$ is a field. Since $R=\Z[W]/I$, we have $S=R[X,Y,Z]=\Z[W,X,Y,Z]/I$ and $\mathbb{F}=S/J=\Z[W,X,Y,Z]/(I+J)$. Finally since $ \mathfrak{F}(\mathbb{P})$ is finite, each set of variables $W,X,Y,Z$ is finite. Then $\mathbb{F}$ is a finite field by Lemma \ref{lem:fgfield}.

Let $\phi: R\rightarrow \mathbb{F}$ be the restriction to $R$ of the natural ring homomorphism $\psi:S\rightarrow S/J=\mathbb{F}$. We verify that $(\mathbb{F},\phi)$ is a proxy for $\mathbb{P}$.
Since $\phi$ is a ring homomorphism, it is necessarily a partial field homomorphism. 
Moreover, $\phi$ is injective on $\mathfrak{F}(\mathbb{P})$, for if $\phi(p)=\phi(q)$ for some distinct $p, q\in \mathfrak{F}(\mathbb{P})$, then we get the contradiction 
$$-1=(\psi(p)-\psi(q))\psi(X_{pq})-1=\psi((p-q)X_{pq}-1)\in \psi(J)=\{0\}.$$
Second, if $p+q\neq 1$ but $\phi(p)+\phi(q)= 1$ then 
$$-1=(\psi(p)+\psi(q)-1)\psi(Y_{pq})-1=\psi((p+q-1)Y_{pq}-1)\in \psi(J)=\{0\},$$
a contradiction. Third, if $\phi(p)\phi(q)\phi(r)= 1$ when $pqr\neq 1$ we get 
$$-1=(\psi(p)\psi(q)\psi(r)-1)\psi(Z_{pqr})-1=\psi((pqr-1)Z_{pqr}-1)\in \psi(J)=\{0\},$$
a contradiction. Finally, if $1\neq -1$ in $\mathbb{P}$ then  $1\neq -1$ in $\mathbb{F}$, for otherwise we get the contradiction $-1=(\psi(1)+\psi(1))\psi(X_{11})-1=\psi(2X_{11}-1)\in \psi(J)=\{0\}$.
\endproof

\section{Implementation details}
\label{impl}

Our implementation of these computations was written using SageMath 8.1, making extensive use of the Matroid Theory library.
Computations were run in a virtual machine on an Intel Xeon E5-2690 v4 64-bit x86 microprocessor operating at 2.6GHz, with 4 cores and 23GB of memory available.

Let $\mathbb{P} \in \{\mathbb{D}, \mathbb{U}_2\}$; 
we want to find excluded minors of size at most $n$ for the class of $\mathbb{P}$-representable matroids $\mathcal{M}(\mathbb{P})$.
Let $\mathcal{N}$ be a set of strong $\mathbb{P}$-stabilizers such that each $N \in \mathcal{N}$ is not near-regular.
In what follows, we use $\widetilde{\mathcal{M}}_\mathcal{N}(\mathbb{P})$ to denote the set of all $3$-connected matroids in $\mathcal{M}(\mathbb{P})$ with an $\mathcal{N}$-minor.

We generate all matroids in $\widetilde{\mathcal{M}}_\mathcal{N}(\mathbb{P})$ of size at most $n$.
To find the excluded minors of size $n$, our basic approach is as follows. 
First, find all $3$-connected extensions of $(n-1)$-element matroids in $\widetilde{\mathcal{M}}_\mathcal{N}(\mathbb{P})$; second, filter out those isomorphic to an $n$-element matroid in $\widetilde{\mathcal{M}}_\mathcal{N}(\mathbb{P})$; finally, filter out those that contain, as a minor, an excluded-minor for $\mathcal{M}(\mathbb{P})$ of size at most $n-1$.

\subsection{Restricting to ternary or quaternary excluded minors}
\label{linrestrict}
As we are dealing with a partial field $\mathbb{P} \in \{\mathbb{D}, \mathbb{U}_2\}$, which has a partial-field homomorphism to either $\GF(3)$ or $\GF(4)$, the efficiency of the first step can be improved using the excluded-minor characterisations for ternary and quaternary matroids.

\begin{lemma}
  \label{onlyquaternary}
  Let $M$ be an excluded minor for the class of $2$-regular matroids.
  If $|E(M)| \ge 9$, then $M$ is 
  quaternary.
\end{lemma}
\begin{proof}
  Suppose $|E(M)| \ge 9$ and, towards a contradiction, that $M$ is not $\GF(4)$-representable.
  Then $M$ has a minor $N$ isomorphic to one of the seven excluded minors for $\GF(4)$ (see \cref{gf4minors}).
  Since each of these excluded minors has at most eight elements, $M$ contains $N$ as a proper minor.  But $M$ is an excluded minor, so $N$ is $2$-regular; a contradiction.
\end{proof}

The following lemma follows, in a similar manner, from the excluded-minor characterisation of ternary matroids.
\begin{lemma}
  \label{onlyternary}
  If $M$ is an excluded minor for dyadic matroids with $|E(M)| \ge 8$, then $M$ is ternary.
\end{lemma}

By \cref{onlyquaternary,onlyternary}, at the first step of our procedure for finding excluded minors, we need only consider ternary or quaternary $3$-connected extensions of $(n-1)$-element matroids in $\widetilde{\mathcal{M}}_\mathcal{N}(\mathbb{P})$.
We can further reduce the number of potential excluded minors to consider using splicing, which we explain in \cref{sec-splic}.

\subsection{Generating \texorpdfstring{$\mathbb{P}$}{P}-representable matroids}
\label{gen-iso}

To simulate generating a $\mathbb{P}$-representable matroid, we use 
partial field proxies,
as described in \cref{preconfine}.
That is, we find a prime $p$, and partial-field homomorphism $\phi : \mathbb{P} \rightarrow \GF(p)$, such that a matroid is $\mathbb{P}$-representable if and only if it has a $\phi(\mathfrak{F}(\mathbb{P}))$-confined representation over $\GF(p)$ (see \cref{thm:proxy} and \cref{fig:conff}).
Then, to find $\mathbb{P}$-representable single-element extensions of a matroid with $\mathbb{P}$-representation $A$, we can find single-element extensions of $\phi(A)$ with a $\GF(p)$-representation whose cross ratios are in $\phi(\mathfrak{F}(\mathbb{P}))$.

For a class $\mathcal{M}(\mathbb{P})$ with a set of strong $\mathbb{P}$-stabilizers $\mathcal{N}$, 
we generate a representative $M$ of each isomorphism class in $\widetilde{\mathcal{M}}_\mathcal{N}(\mathbb{P})$ consisting of matroids of size at most $n$.
%

Suppose we have generated all matroids in $\widetilde{\mathcal{M}}_\mathcal{N}(\mathbb{P})$ of size at most $n-1$ (up to isomorphs).
Initially, if $n_0$ is the size of the smallest matroid in $\mathcal{N}$, then $n = n_0 +1$.
Let $M[A]$ be a $\mathbb{P}$-represented matroid.
We say that the $\mathbb{P}$-represented matroid $M[A | e]$, for some column vector $e$ with entries in $\mathbb{P}$, is a \emph{linear extension} of $M[A]$.
For each $(n-1)$ element $\mathbb{P}$-represented matroid, we generate all simple linear extensions (where the representations have the appropriate cross ratios; this functionality is provided by the function 
\texttt{LinearMatroid.linear\char`_extensions()} in SageMath).
Note that each of these simple matroids is in fact $3$-connected (by \cite[Proposition~8.2.7]{oxley}).
After closing this set under duality, and adding any $n$-element matroid in $\mathcal{N}$, the set consists of all $n$-element matroids in $\widetilde{\mathcal{M}}_\mathcal{N}(\mathbb{P})$,
%
by \cref{seysplitcorr} and
since each matroid in $\mathcal{N}$ is a strong $\mathbb{P}$-stabilizer.


\subsection{Isomorph filtering}
\label{sec-iso}

We use an isomorphism invariant, which can be efficiently computed, to distinguish matroids that can be easily identified as non-isomorphic.  Two matroids with different values for the invariant are non-isomorphic; whereas two matroids with the same value for the invariant require a full isomorphism check.
The isomorphism invariant we use is provided by the function \verb|BasisMatroid._bases_invariant()| in SageMath, and is based on the incidences of groundset elements with bases.
%

As $n$ increases, we have to deal with more matroids than can be loaded in memory at once.
Thus, to filter isomorphic matroids, we use a batched two-pass approach.
We consider the matroids in batches of an appropriate size so that an entire batch can be kept in memory at once.  
First, batch by batch, we compute a hash of the matroid invariant 
for each matroid in the batch, and write the matroids to disk, stored in $g$ groups, grouped by the hash modulo~$g$.
(The value of $g$ is chosen to ensure all matroids in a group can also be loaded in memory at once.)
Call the hash of the invariant the \emph{raw hash}, and call the hash modulo~$g$ the \emph{hash~mod}.
Then, in turn, we load each of the $g$ groups; that is, for each $i \in \{0,1,\dotsc,g-1\}$, we load all matroids whose hash~mod is $i$.  Within each group, isomorphs are filtered by isomorphism checking those matroids with the same raw hash.

\subsection{Minor checking}
\label{sec-minor}

Let $M$ and $N$ be matroids.
To check if $M$ has a minor isomorphic to $N$, we use a simple approach that avoids repetitive computations.
If $|E(N)| = |E(M)|$, then we check if $N$ is isomorphic to $M$; otherwise, for each single-element deletion and contraction of $M$, we recursively check if any of these matroids has an $N$-minor.
However, we cache the result of each minor check (keyed by the isomorphism class), and use cached results when available, to avoid repetition.
Full isomorphism checking is performed only when the isomorphism invariants match, as described in \cref{sec-iso}.

\subsection{Splicing}
\label{sec-splic}


Let $M'$ be a matroid, let $M_e$ be a single-element extension of $M'$ by an element~$e$, and let $M_f$ be a single-element extension of $M'$ by an element~$f$, where $e$ and $f$ are distinct.
Note that $M_e$ and $M_f$ may be isomorphic.
We say that $M$ is a \emph{splice of $M_e$ and $M_f$} if $M \ba e = M_f$ and $M \ba f = M_e$.


Suppose we wish to find the excluded minors of size $n$ for the class 
$\mathcal{M}(\mathbb{P})$.
In order to reduce the number of matroids to consider as potential excluded minors, rather than
generating all extensions of $(n-1)$-element matroids in $\widetilde{\mathcal{M}}_\mathcal{N}(\mathbb{P})$, we can instead generate splices of each pair of $(n-1)$-element matroids in $\widetilde{\mathcal{M}}_\mathcal{N}(\mathbb{P})$ that are extensions of some $(n-2)$-element matroid in $\widetilde{\mathcal{M}}_\mathcal{N}(\mathbb{P})$.
Note that the two matroids in such a pair may be isomorphic.
In order for this splicing process to be exhaustive, we require a guarantee that for any excluded minor $M$, there is (up to duality) some pair $e,f \in E(M)$ such that $M \ba e$, $M \ba f$, and $M \ba e \ba f$ are in $\widetilde{\mathcal{M}}_\mathcal{N}(\mathbb{P})$.  \Cref{detachthm} is such a guarantee when $M$ does not contain any spikey $3$-separators.  We work towards showing that spikey $3$-separators do not appear in an excluded minor $M$ when $M$ is large.

First, there is a subtlety worth noting.
Let $M_x$ and $M'$ be matroids with $E(M_x) = E(M') \cup \{x\}$, and
suppose $M' \cong M_x \ba x$.  Clearly $M'$ has a single-element extension, by an element $x$, that is isomorphic to $M_x$, but there may be more than one distinct extensions with this property, due to automorphisms of $M_x$.
To obtain all splices, 
it is not enough to consider just one of these extensions. 
%
For each $(n-2)$-element 
matroid $M' \in \widetilde{\mathcal{M}}_\mathcal{N}(\mathbb{P})$, and each
$(n-1)$-element matroid $M_x \in \widetilde{\mathcal{M}}_\mathcal{N}(\mathbb{P})$ such that $M_x \ba x \cong M'$ for some $x \in E(M_x)$,
we keep track of all single-element extensions of $M'$ to a matroid isomorphic to $M_x$; denote these extensions as $\mathcal{X}(M_x)$.  We also maintain, for each matroid $X \in \mathcal{X}(M_x)$, the isomorphism between $M_x \ba x$ and $X\ba x$.
Using this information, for each matroid $M'$,
and each (possibly isomorphic) pair $\{M_e,M_f\} \subseteq \widetilde{\mathcal{M}}_\mathcal{N}(\mathbb{P})$ such that $M_x \ba x \cong M'$ for $x \in \{e,f\}$, and each $X_e \in \mathcal{X}(M_e)$ and $X_f \in \mathcal{X}(M_f)$, we compute the splice of $X_e$ and $X_f$.
For simplicity, we refer to the set of all of these matroids as ``the splices of $M_e$ and $M_f$''.

The following generalises \cite[Lemma 7.2]{BCOSW2018}; as the proof is similar, we provide only a sketch.



\begin{lemma}
  \label{spikeys}
  Let $\mathbb{P}$ be a partial field, let $N$ be a non-binary $3$-connected strong $\mathbb{P}$-stabilizer, and let $M$ be an excluded minor for $\mathcal{M}(\mathbb{P})$, where $M$ has an $N$-minor.
  If $M$ has a \spikey\ $P$ such that at most one element of $E(M)-E(N)$ is not in $P$, then $|E(M)| \le |E(N)| + 5$.
\end{lemma}
\begin{proof}
  Since at most one element of $E(M)-E(N)$ is not in $P$, we have that $|P-E(N)| \ge 5$.
  By dualising, if necessary, we may assume that there are distinct elements $a,b \in P$ such that $M \ba a \ba b$ has an $N$-minor, with $a \in K_i$ and $b \in K_j$ for $i \neq j$, where $\{K_1,\dotsc,K_t\}$ is a partition of $P$ such that $K_{i'} \cup K_{j'}$ is a circuit for all distinct $i',j' \in \{1,\dotsc,t\}$.
  Now $M \ba a$, $M \ba b$ and $\co(M\ba a \ba b)$ are $3$-connected.

  By the definition of a \spikey, the pair $\{a,b\}$ is contained in a $4$-element cocircuit~$C^* \subseteq P$.
  Let $u \in C^*-\{a,b\}$.
  Then $u$ is in a series pair of $M \ba a \ba b$, so $M \ba a \ba b / u$ has an $N$-minor, and $\co(M \ba a \ba b/u)$ is $3$-connected.
  Moreover, $M / u$ is $3$-connected.
  The result then follows using the same argument as in \cite[Lemma 7.2]{BCOSW2018}.
\end{proof}

\begin{lemma}
  \label{splicinglemma}
  Let $\mathbb{P}$ be a partial field, and let $\mathcal{N}$ be a set of non-binary strong $\mathbb{P}$-stabilizers for $\mathcal{M}(\mathbb{P})$.
  Let $M$ be an excluded minor for $\mathcal{M}(\mathbb{P})$ such that $M$ has an $\mathcal{N}$-minor, $|E(M)| \ge 13$, and $|E(M)| \ge |E(N)| + 6$ for each $N \in \mathcal{N}$.
  Then there is a matroid $M'$ that is $\Delta Y$-equivalent to $M$ or $M^*$, and distinct elements $e,f \in E(M')$ such that for each $M'' \in \{M' \ba e \ba f, \, M' \ba e, \, M' \ba f\}$, the matroid $M''$ is $3$-connected, has an $\mathcal{N}$-minor, and $M'' \in \mathcal{M}(\mathbb{P})$.
\end{lemma}
\begin{proof}
  Let $N \in \mathcal{N}$ such that $M$ has an $N$-minor.
  By \cref{detachthm}, either there exists a matroid $M'$ that is $\Delta Y$-equivalent to $M$ or $M^*$ and a pair of elements $\{e,f\}$ such that either $M' \ba e \ba f$ is $3$-connected with an $N$-minor, or $M'$ has a \spikey~$P$.
  In the latter case, as $|E(M)| \ge 13$ there is at most one element of $E(M)-E(N)$ is not in $P$, so, by \cref{spikeys}, $|E(M)| \le |E(N)| + 5$; a contradiction.
  We deduce that there is a pair $\{e,f\}$ such that $M'\ba e \ba f$ is $3$-connected with an $N$-minor.
  It follows that $M'\ba e$ and $M'\ba f$ are $3$-connected with an $N$-minor.
  Moreover, since $M'$ is an excluded minor for the class $\mathcal{M}(\mathbb{P})$, by \cref{deltaYexc}, each of $M' \ba e$, $M'\ba f$, and $M'\ba e\ba f$ is in $\mathcal{M}(\mathbb{P})$.
\end{proof}

As described in \cref{linrestrict},
when $\mathbb{P} = \mathbb{D}$ or $\mathbb{P} = \mathbb{U}_2$,
we may restrict our attention to ternary or quaternary excluded minors respectively; so it suffices to find splices that are ternary or quaternary, respectively.



\subsection{Testing}

Implementations were tested before use.  In particular, the excluded-minor computation routines were checked using the known characterisation for $\GF(4)$ \cite{GGK2000}, and using the known excluded minors for $\GF(5)$-representable matroids on up to $9$ elements \cite{MR2008}.  The excluded minors for dyadic matroids on up to 13 elements have previously been computed by Pendavingh; our results were also consistent with those.
Regarding the generation of matroids in $\mathcal{M}(\mathbb{P})$, the matroids that we generated were consistent with known maximum-sized $\mathbb{P}$-representable matroids for $\mathbb{P} \in \{\mathbb{D}, \mathbb{U}_2\}$ \cite{Kung1990,Kung1988,Semple1998}. 
Our splicing implementation was tested by independently generating all (ternary/quaternary) matroids in $\widetilde{\mathcal{M}}_\mathcal{N}(\mathbb{P})$ with a pair $\{x,y\}$ such that $M\ba x \ba y \in \widetilde{\mathcal{M}}_\mathcal{N}(\mathbb{P})$, and ensuring that these are precisely the matroids obtained by splicing.

\section{Dyadic matroids}
\label{secdyadic}

In this section we present the results of the computation of the excluded minors for dyadic matroids on at most 15 elements.
The next lemma is a consequence of \cref{nr_exminors}, and the subsequent lemma is well known and easy to verify (see \cite[Proposition~3.1]{GOVW1998}, for example).
\begin{lemma}
  \label{nou25u35dy}
  Let $M$ be an excluded minor for the class of dyadic matroids.
  Then, either
  \begin{enumerate}
    \item $M$ 
        has a $\{F_7^-, (F_7^-)^*, P_8\}$-minor, or
    \item $M$ 
        is isomorphic to one of $U_{2,5}$, $U_{3,5}$, $F_7$, $F_7^*$, $\AG(2,3)\ba e$, $(\AG(2,3)\ba e)^*$, and $(\AG(2, 3)\ba e)^{\Delta Y}$.
  \end{enumerate}
\end{lemma}

\begin{lemma}
  The matroids $F_7^-$, $(F_7^-)^*$, and $P_8$ are strong $\mathbb{D}$-stabilizers. 
\end{lemma}

The excluded minors for dyadic matroids are known to include 
the seven matroids listed in \cref{nou25u35dy}(ii),
as well as an $8$-element matroid known as $T_8$, a $10$-element matroid known as $N_{1}$, and a $12$-element matroid known as $N_2$ (see \cite[Problem~14.7.11]{oxley}).

We computed an exhaustive list of the excluded minors on at most $15$ elements, finding one more, previously unknown, excluded minor, on $14$ elements.
This matroid, which we call $N_3$, has a reduced $\GF(3)$-representation as follows:
$$\begin{bmatrix}
 1 & 2 & 0 & 0 & 1 & 2 & 2 \\
 2 & 2 & 2 & 0 & 1 & 1 & 2 \\
 0 & 2 & 0 & 0 & 1 & 1 & 2 \\
 0 & 0 & 0 & 0 & 2 & 1 & 2 \\
 1 & 1 & 1 & 2 & 1 & 2 & 2 \\
 2 & 1 & 1 & 1 & 2 & 1 & 1 \\
 2 & 2 & 2 & 2 & 2 & 1 & 0
\end{bmatrix}$$

\begin{theorem}
  \label{dyexprop}
  The excluded minors for dyadic matroids on at most $15$ elements are
  $U_{2,5}$, $U_{3,5}$, $F_7$, $F_7^*$, $\AG(2,3)\ba e$, $(\AG(2,3)\ba e)^*$, $(\AG(2,3)\ba e)^{\Delta Y}$, $T_8$, $N_{1}$, $N_2$, and $N_{3}$.
\end{theorem}
\begin{proof}
  We exhaustively generated all $n$-element dyadic matroids that are not near-regular for $n \le 15$; see \cref{dytable}.

  By \cref{nou25u35dy}, the excluded minors on at most seven elements are $U_{2,5}$, $U_{3,5}$, $F_7$, and $F_7^*$.
  Let $8 \le n \le 14$, and suppose all excluded minors for dyadic matroids on fewer than $n$ elements are known.
  We generated all matroids that are ternary single-element extensions of some $(n-1)$-element dyadic matroid with a $\{F_7^-,(F_7^-)^*,P_8\}$-minor.
  From this list of potential excluded minors, we first filtered out those in our list of $n$-element dyadic matroids, and then also filtered out any matroids that contained, as a minor, any of the excluded minors for dyadic matroids on fewer than $n$ elements.  Each remaining matroid 
  is an excluded minor.
  On the other hand, if $M$ is an $n$-element excluded minor not listed in \cref{nou25u35dy}(ii), then, by \cref{nou25u35dy,onlyternary,seysplitcorr}, this collection of generated matroids contains at least one of $M$ and $M^*$.

  Now let $n = 15$, and again suppose all excluded minors on fewer than $n$ elements are known.
  We generated all $3$-connected ternary 
  splices of a (not-necessarily non-isomorphic) pair of $(n-1)$-element dyadic matroids that are each single-element extensions of an $(n-2)$-element $3$-connected dyadic matroid with a $\{F_7^-,(F_7^-)^*,P_8\}$-minor; call this collection of generated matroids $\mathcal{S}$.
  Since $n \ge |E(P_8)| + 6 = 14$, \cref{splicinglemma} implies that if $M$ is an $n$-element excluded minor, then, for some $M' \in \Delta^{(*)}(M)$, there exists a pair $\{e,f\} \subseteq E(M')$ such that $M' \ba e$, $M' \ba f$, and $M' \ba \{e,f\}$ are $3$-connected and have a $\{F_7^-,(F_7^-)^*,P_8\}$-minor.
  Thus $M' \in \mathcal{S}$.
  (For reference, $\mathcal{S}$ contained 20632781 pairwise non-isomorphic $15$-element rank-$7$ matroids, and 8840124 pairwise non-isomorphic $15$-element rank-$8$ matroids.)
  As before, from this list of potential excluded minors, we filtered out those matroids that were dyadic or contained, as a minor, any of the excluded minors for dyadic matroids on fewer than $n$ elements.
\end{proof}


\begin{table}[htb]
  \begin{tabular}{r|r r r r r r r r r r}
    \hline
$r \ba n$ & 7 & 8 &  9 &  10 &   11 &    12 &     13 &     14 &      15 \\
    \hline                                                              
        3 & 1 & 1 &  1 &     &      &       &        &        &         \\
        4 & 1 & 7 & 24 &  52 &   60 &    44 &     20 &      7 &       2 \\
        5 &   & 1 & 24 & 223 & 1087 &  3000 &   5065 &   5651 &    4553 \\
        6 &   &   &  1 &  52 & 1087 & 10755 &  57169 & 185354 &  398875 \\
        7 &   &   &    &     &   60 &  3000 &  57169 & 540268 & 2986648 \\ 
        8 &   &   &    &     &      &    44 &   5065 & 185354 & 2986648 \\
        9 &   &   &    &     &      &       &     20 &   5651 &  398875 \\ 
       10 &   &   &    &     &      &       &        &      7 &    4553 \\
       11 &   &   &    &     &      &       &        &        &       2 \\
       \hline                                                 
    Total & 2 & 9 & 50 & 327 & 2294 & 16843 & 124508 & 922292 & 6780156 \\
       \hline
  \end{tabular}
  \caption{The number of $3$-connected $n$-element rank-$r$ dyadic matroids with a $\{F_7^-,(F_7^-)^*,P_8\}$-minor, for $n \le 15$.}
  \label{dytable}
\end{table}

It turns out that the list of matroids in \cref{dyexprop} is not the complete list of excluded minors for dyadic matroids.
We also found an excluded minor with 16 elements; we call this matroid $N_4$.
The following is a reduced $\GF(3)$-representation of $N_4$:
$$\begin{bmatrix}
 1 & 0 & 1 & 1 & 1 & 1 & 2 & 1 \\
 0 & 2 & 0 & 0 & 1 & 0 & 0 & 1 \\
 1 & 0 & 2 & 1 & 0 & 1 & 2 & 1 \\
 1 & 0 & 1 & 0 & 0 & 0 & 1 & 0 \\
 1 & 1 & 0 & 0 & 0 & 1 & 0 & 0 \\
 1 & 0 & 1 & 0 & 1 & 1 & 0 & 1 \\
 2 & 0 & 2 & 1 & 0 & 0 & 2 & 1 \\
 1 & 1 & 1 & 0 & 0 & 1 & 1 & 0
\end{bmatrix}$$

We found this matroid by a computer search, as follows.
Observe that the matroids $T_8$, $N_1$, $N_2$, and $N_3$ are self-dual matroids on 8, 10, 12, and 14 elements respectively, and each has a pair of disjoint circuit-hyperplanes.
Starting with the 2986648 $3$-connected rank-8 dyadic non-near-regular matroids on 15 elements, 285488 of these matroids have a circuit-hyperplane whose complement is independent.  Of these, 4875 have at least one $3$-connected ternary extension to a matroid with disjoint circuit-hyperplanes.  There are 288076 such matroids, but 52 are dyadic and 288023 properly contain an excluded minor for dyadic matroids.  The one other matroid is $N_4$.

\begin{table}[htb]
  \begin{tabular}{ c c c c }
    \hline
    $M$ & $\mathbb{P}_M$ & 
                                               $\left|\Delta(M)\right|$ \\
    \hline
    $U_{2,5}$        & $\mathbb{U}_2$ & 
                                               2 \\
    $F_7$            & $\GF(2)$       & 
                                               2 \\
    $\AG(2,3) \ba e$ & $\mathbb{S}$   & 
                                               3 \\
    $T_8$            & $\GF(3)$       & 
                                               1 \\
    $N_1$            & $\GF(3)$       & 
                                               1 \\
    $N_2$            & $\GF(3)$       & 
                                               1 \\
    $N_3$            & $\GF(3)$       & 
                                               1 \\
    $N_4$            & $\GF(3)$ 
                                      & 
                                               1 \\
    \hline
  \end{tabular}
  \caption{Excluded minors for the class of dyadic matroids, and their universal partial fields. 
    We list one representative~$M$ of each $\Delta Y$-equivalence class $\Delta(M)$.%
  }
  \label{dyupfs}
\end{table}


Finally, using \cref{onlyternary,dyexprop}, we observe that with the exception of $U_{2,5}$ and $U_{3,5}$, each excluded minor for the class of dyadic matroids is not $\GF(5)$-representable, so is an excluded minor for the class of $\GF(5)$-representable matroids.
In \cref{dyupfs}, we provide the universal partial field for each of the known excluded minors.  The matroids with universal partial field $\GF(3)$ are representable only over fields with characteristic three.

\section{\texorpdfstring{$2$}{2}-regular matroids}
\label{sec2reg}

We now present the results of the computation of the excluded minors for $2$-regular matroids on at most 15
elements.
The next lemma is a consequence of \cite[Lemmas~5.7 and~5.25]{OSV2000}.

\begin{lemma}
  The matroids $U_{2,5}$ and $U_{3,5}$ are strong $\mathbb{U}_2$-stabilizers. 
\end{lemma}




\begin{lemma}
  \label{nou25u35}
  Let $M$ be an excluded minor for the class of 
  $2$-regular matroids.
  Then, either
  \begin{enumerate}
    \item $M$ has a $\{U_{2,5}, U_{3,5}\}$-minor, or
    \item $M$ is isomorphic to one of $F_7$, $F_7^*$, $F_7^-$, $(F_7^-)^*$, $\AG(2,3)\ba e$, $(\AG(2,3)\ba e)^*$, $(\AG(2, 3)\ba e)^{\Delta Y}$, and $P_8$.
  \end{enumerate}
\end{lemma}
\begin{proof}
  Suppose that $M$ has no $\{U_{2,5},U_{3,5}\}$-minor.
  Since $M$ is not, in particular, near-regular, \cref{nr_exminors} implies that
  $M$ has a minor isomorphic to one of $F_7$, $F_7^*$, $F_7^-$, $(F_7^-)^*$, $\AG(2,3)\ba e$, $(\AG(2,3)\ba e)^*$, $(\AG(2,3)\ba e)^{\Delta Y}$, and $P_8$.

  It is well known that $F_7$ and $F_7^*$ are representable over a field $\mathbb{F}$ if and only if $\mathbb{F}$ has characteristic two; whereas $F_7^*$, $(F_7^-)^*$, and $P_8$ are representable over a field $\mathbb{F}$ if and only if $\mathbb{F}$ does not have characteristic two.
  Moreover, $\AG(2,3) \ba e$ is not $\GF(5)$-representable \cite[Proposition~7.3]{HMvZ2011}, and hence $(\AG(2,3)\ba e)^*$ and $(\AG(2,3)\ba e)^{\Delta Y}$ are also not $\GF(5)$-representable, the latter by \cref{deltaYrep}.
  Since each of these eight matroids is not representable over 
  either $\GF(4)$ or $\GF(5)$, we 
  deduce that $M$ does not contain one of these matroids as a proper minor, so (ii) holds, as required.
\end{proof}

By \cref{nou25u35}, in our search for excluded minors for the class of $2$-regular matroids, we can restrict our focus to matroids with a $\{U_{2,5},U_{3,5}\}$-minor.
%
The matroids $U_{2,6}$, $U_{4,6}$, $P_6$, $P_8$, and $P_8^=$ are not $2$-regular, as they are not $\GF(4)$-representable, 
by \cref{gf4minors}.
Let $F_7^=$ denote the matroid obtained by relaxing a circuit-hyperplane of the non-Fano matroid $F_7^-$, as illustrated in \cref{fanosfig}.
%
%
Recall that $P_8^=$ is obtained from $P_8$ by relaxing disjoint circuit-hyperplanes; let $P_8^-$ denote the matroid obtained by relaxing just one of a pair of disjoint circuit-hyperplanes of $P_8$.
It is known that 
$U_{3,6}$, $F_7^=$ and $(F_7^=)^*$ are not $2$-regular \cite[Lemmas~4.2.4 and~4.2.5]{Semple1998}; and neither is
$P_8^-$ \cite[Section 4.1]{COvZ2018}.
It turns out that all these matroids are excluded minors for the class of $2$-regular matroids. 

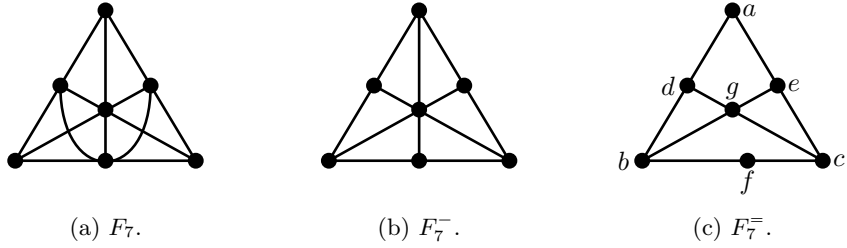
\begin{figure}
  \begin{subfigure}{0.32\textwidth}
    \centering
    \begin{tikzpicture}[rotate=90,yscale=0.8,line width=1pt]
      \tikzset{VertexStyle/.append style = {minimum height=5,minimum width=5}}
      \clip (-0.5,-4.5) rectangle (2.5,-.5);
      \draw (0,-1) -- (2,-2.5) -- (0,-4);
      \draw (1,-1.75) -- (0,-4);
      \draw (1,-3.25) -- (0,-1);
      \draw (0,-1) -- (0,-4);
      \draw (2,-2.5) -- (0,-2.5);
      \draw (1,-3.25) .. controls (-.35,-3.25) and (-.35,-1.75) .. (1,-1.75);

      \SetVertexNoLabel
      \Vertex[x=2,y=-2.5]{a2}
      \Vertex[x=0.68,y=-2.5]{a3}
      \Vertex[x=1,y=-3.25]{a4}
      \Vertex[x=1,y=-1.75]{a5}
      \Vertex[x=0,y=-2.5]{d}
      \Vertex[x=0,y=-1]{e}
      \Vertex[x=0,y=-4]{f}
    \end{tikzpicture}
    \caption{$F_7$.}
  \end{subfigure}
  \begin{subfigure}{0.32\textwidth}
    \centering
    \begin{tikzpicture}[rotate=90,yscale=0.8,line width=1pt]
      \tikzset{VertexStyle/.append style = {minimum height=5,minimum width=5}}
      \clip (-0.5,-4.5) rectangle (2.5,-.5);
      \draw (0,-1) -- (2,-2.5) -- (0,-4);
      \draw (1,-1.75) -- (0,-4);
      \draw (1,-3.25) -- (0,-1);
      \draw (0,-1) -- (0,-4);
      \draw (2,-2.5) -- (0,-2.5);

      \SetVertexNoLabel
      \Vertex[x=2,y=-2.5]{a2}
      \Vertex[x=0.68,y=-2.5]{a3}
      \Vertex[x=1,y=-3.25]{a4}
      \Vertex[x=1,y=-1.75]{a5}
      \Vertex[x=0,y=-2.5]{d}
      \Vertex[x=0,y=-1]{e}
      \Vertex[x=0,y=-4]{f}
    \end{tikzpicture}
    \caption{$F_7^-$.}
  \end{subfigure}
  \begin{subfigure}{0.32\textwidth}
    \centering
    \begin{tikzpicture}[rotate=90,yscale=0.8,line width=1pt]
      \tikzset{VertexStyle/.append style = {minimum height=5,minimum width=5}}
      \clip (-0.5,-4.5) rectangle (2.5,-.5);
      \draw (0,-1) -- (2,-2.5) -- (0,-4);
      \draw (1,-1.75) -- (0,-4);
      \draw (1,-3.25) -- (0,-1);
      \draw (0,-1) -- (0,-4);


      \Vertex[L=$a$,Lpos=0,LabelOut=true,x=2,y=-2.5]{a2}
      \Vertex[L=$g$,Lpos=90,LabelOut=true,x=0.68,y=-2.5]{a3}
      \Vertex[L=$e$,Lpos=0,LabelOut=true,x=1,y=-3.25]{a4}
      \Vertex[L=$d$,Lpos=180,LabelOut=true,x=1,y=-1.75]{a5}
      \Vertex[L=$f$,Lpos=-90,LabelOut=true,x=0,y=-2.75]{d}
      \Vertex[L=$b$,Lpos=180,LabelOut=true,x=0,y=-1]{e}
      \Vertex[L=$c$,Lpos=0,LabelOut=true,x=0,y=-4]{f}
    \end{tikzpicture}
    \caption{$F_7^=$.}
  \end{subfigure}
  \caption{Three of the excluded minors for $2$-regular matroids.}
  \label{fanosfig}
\end{figure}

\begin{figure}
  \begin{tikzpicture}[scale=0.39,line width=1pt]
    \tikzset{VertexStyle/.append style = {minimum height=5,minimum width=5}}
    \draw (-8.5,0.6) -- (6.5,0.6) -- (8.93,6.5) -- (-5.97,6.5) -- (-8.5,0.6);
    \draw (-8.5,0.6) -- (-8.5,-5) -- (6.5,-5) -- (6.5,0.6);

    \Vertex[Lpos=135,LabelOut=true,L=$d$,x=-0.4,y=3.1]{v1}
    \Vertex[Lpos=-85,LabelOut=true,L=$c$,x=1,y=2.3]{v4}
    \Vertex[LabelOut=true,L=$b$,x=3.2,y=3.2]{v6}
    \Vertex[LabelOut=true,L=$a$,x=2.4,y=5.7]{v7}

    \Vertex[Lpos=-60,LabelOut=true,L=$f$,x=-1.2,y=-1.6]{v2}
    \Vertex[Lpos= 85,LabelOut=true,L=$g$,x=-3.0,y=-1.1]{v0}
    \Vertex[Lpos=180,LabelOut=true,L=$h$,x=-4.5,y=-1.9]{v3}
    \Vertex[Lpos=0,LabelOut=true,L=$e$,x=-3.0,y=-4.3]{v5}
    \SetVertexNoLabel
    \tikzset{VertexStyle/.append style = {shape=rectangle,fill=white}}
    \Vertex[x=-6,y=0.6]{u0}
    \Vertex[x=-3,y=0.6]{u1}
    \Vertex[x=0.3,y=0.6]{u2}
    \Vertex[x=4,y=0.6]{u3}

    \Edge(u0)(v5)
    \Edges(v5,u1,v7,u3)
    \Edges(v5,u2,v7)
    \Edges(u1,v6)
    \Edges(u2,v3)

    \Edge(v1)(u3)
    \tikzset{EdgeStyle/.append style = {bend right=18}}
    \Edge(v6)(u0)
    \tikzset{EdgeStyle/.append style = {bend right=10}}
    \Edge(u0)(v2)

    \tikzset{EdgeStyle/.append style = {bend left=15}}
    \Edge(u3)(v3)
  \end{tikzpicture}
  \caption{$P_8$, an excluded minor for $2$-regular matroids.  Relaxing $\{e,f,g,h\}$ results in the matroid $P_8^-$; relaxing both $\{a,b,c,d\}$ and $\{e,f,g,h\}$ results in the matroid $P_8^=$.}
  \label{p8fig}
\end{figure}
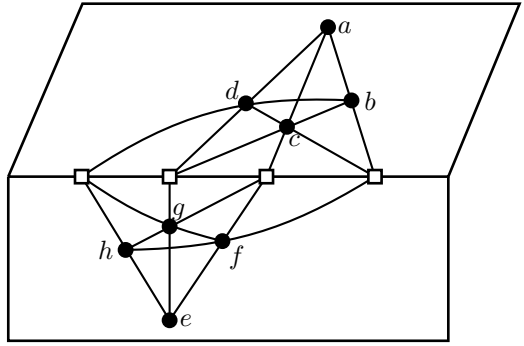

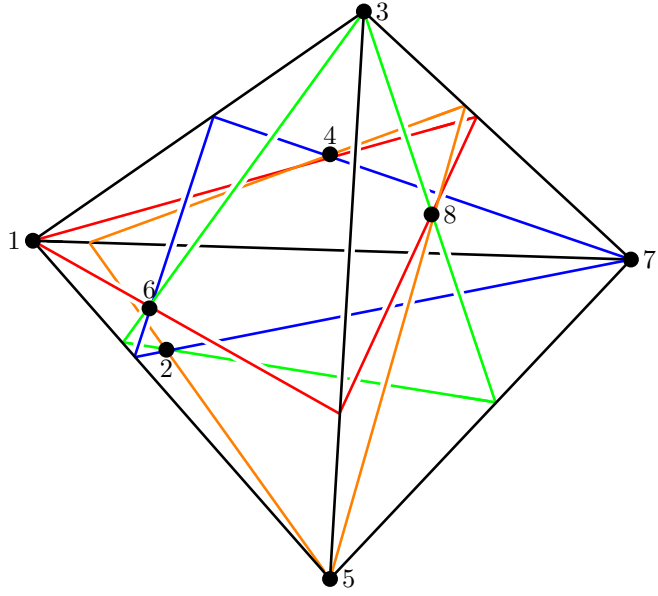
\begin{figure}
  \centering
  \begin{tikzpicture}[rotate=0,scale=0.5,line width=1pt]
    \tikzset{VertexStyle/.append style = {minimum height=5,minimum width=5}}
    \draw (13.4,-1.3) -- (-2.5,-0.8);

    \draw[color=green] (9.8,-5.1) -- (-.1,-3.5);
    \draw[color=orange] (-1,-0.85) -- (5.4,-9.8);
    \draw[color=red] (-2.5,-0.8) -- (9.3,2.5);
    \draw[color=orange] (9,2.8) -- (-1,-0.85);
    \draw[color=white,line width=4pt] (2.3,2.5) -- (0.2,-3.9);
    \draw[color=blue] (0.2,-3.9) -- (13.4,-1.3) -- (2.3,2.5);
    \draw[color=white,line width=4pt] (-.1,-3.5) -- (6.3,5.3) -- (9.8,-5.1);
    \draw[color=green] (-.1+.198,-3.5-.032) -- (-.1,-3.5); 
    \draw[color=green] (9.8,-5.1) -- (9.8-.99,-5.1+.16); 
    \draw[color=blue] (2.3,2.5) -- (0.2,-3.9);

    \draw[color=white,line width=4pt] (5.4,-9.8) -- (6.3,5.3);
    \draw[color=white,line width=4pt] (9.3,2.5) -- (5.65,-5.4) -- (-2.5,-0.8);
    \draw[color=red] (9.3-1.18,2.5-.33) -- (9.3,2.5); 
    \draw[color=white,line width=4pt] (5.4,-9.8) -- (9,2.8);
    \draw[color=orange] (9,2.8) -- (9-1.0,2.8-.365); 
    \draw[color=green] (-.1,-3.5) -- (6.3,5.3) -- (9.8,-5.1);
    \draw (-1.705, -0.825) -- (-2.5,-0.8); 
    \draw[color=red] (9.3,2.5) -- (5.65,-5.4) -- (-2.5,-0.8);
    \draw[color=orange] (5.4,-9.8) -- (9,2.8);
    \draw (-2.5,-0.8) -- (6.3,5.3) -- (13.4,-1.3);
    \draw (5.4,-9.8) -- (6.3,5.3);
    \draw (-2.5,-0.8) -- (5.4,-9.8) -- (13.4,-1.3);

    \Vertex[L=$1$,Lpos=180,LabelOut=true,x=-2.5,y=-0.8]{a1}
    \Vertex[L=$3$,LabelOut=true,x=6.3,y=5.3]{a3}
    \Vertex[L=$5$,LabelOut=true,x=5.4,y=-9.8]{a5}
    \Vertex[L=$7$,LabelOut=true,x=13.4,y=-1.3]{a7}

    \Vertex[L=$2$,Lpos=-90,LabelOut=true,x=1.05,y=-3.7]{a2}
    \Vertex[L=$4$,Lpos=90,LabelOut=true,x=5.4,y=1.5]{a4}
    \Vertex[L=$6$,Lpos=90,LabelOut=true,x=.6,y=-2.6]{a6}
    \Vertex[L=$8$,LabelOut=true,x=8.1,y=-.1]{a8}
  \end{tikzpicture}
  \caption{$\TQ_8$, another excluded minor for $2$-regular matroids.}
  \label{tp8fig}
\end{figure}

There is one more excluded minor for the class, that we now describe.
We denote this matroid $\TQ_8$, and let $E(\TQ_8) = \{0,1,\dotsc,7\}$.
The matroid $\TQ_8$ is a rank-$4$ sparse paving matroid with eight non-spanning circuits
$\big\{\{i, i+2, i+4, i+5\} : i \in \{0,1,\dotsc,7\}\big\}$, working modulo 8.
It is illustrated in \cref{tp8fig}.

\begin{theorem}
  \label{2regexminors}
  The excluded minors for $2$-regular matroids on at most $15$ elements are
$U_{2,6}$, $U_{3,6}$, $U_{4,6}$, $P_6$,
$F_7$, $F_7^*$, $F_7^-$, $(F_7^-)^*$, $F_7^=$, $(F_7^=)^*$,
$\AG(2,3)\ba e$, $(\AG(2,3)\ba e)^*$, $(\AG(2,3)\ba e)^{\Delta Y}$, $P_8$, $P_8^-$, $P_8^=$, and $\TQ_8$.
\end{theorem}
\begin{proof}
  We exhaustively generated all $n$-element $2$-regular matroids with a $\{U_{2,5},U_{3,5}\}$-minor for $n \le 15$; see \cref{2regtable}.

  By \cref{nou25u35}, any excluded minor has at least six elements.
  Let $6 \le n \le 15$, and suppose all excluded minors for $2$-regular matroids on fewer than $n$ elements are known.
  For $6 \le n \le 8$, we generated all single-element extensions of some $(n-1)$-element $2$-regular matroid with a $\{U_{2,5},U_{3,5}\}$-minor.
  By \cref{nou25u35,seysplitcorr}, if $M$ is an $n$-element excluded minor not listed in \cref{nou25u35}(ii), then this collection of generated matroids contains at least one of $M$ and $M^*$.
 %
  For $8 < n \le 13$, we generated all matroids that are quaternary single-element extensions of some $(n-1)$-element $2$-regular matroid with a $\{U_{2,5},U_{3,5}\}$-minor.
  For each of these potential excluded minors, we filtered out any matroids in the list of generated $2$-regular matroids, or any matroid containing, as a minor, one of the excluded minors for $2$-regular matroids on fewer than $n$ elements.
  Any matroid remaining after this process is an excluded minor.
  On the other hand, if $M$ is an $n$-element excluded minor not listed in \cref{nou25u35}(ii), then, by \cref{nou25u35,onlyquaternary,seysplitcorr}, the collection of generated potential excluded minors contains at least one of $M$ and $M^*$.

  Finally, let $n \in \{14,15\}$.
  We generated all $3$-connected quaternary 
  splices of a (not-necessarily non-isomorphic) pair of $(n-1)$-element $2$-regular matroids that are each single-element extensions of an $(n-2)$-element $3$-connected $2$-regular matroid with a $\{U_{2,5},U_{3,5}\}$-minor; call this collection of generated matroids $\mathcal{S}$.
  By \cref{splicinglemma}, if $M$ is an $n$-element excluded minor not listed in \cref{nou25u35}(ii), then, for some $M' \in \Delta^{(*)}(M)$, there exists a pair $\{e,f\} \subseteq E(M')$ such that $M' \ba e$, $M' \ba f$, and $M' \ba \{e,f\}$ are $3$-connected and have a $\{U_{2,5},U_{3,5}\}$-minor.
  Thus $M' \in \mathcal{S}$.
  (For reference, $\mathcal{S}$ consisted of 29383778 pairwise non-isomorphic $15$-element rank-$7$ matroids, and 12949820 pairwise non-isomorphic $15$-element rank-$8$ matroids.)
  As before, for each such potential excluded minor $M'$, we filtered out $M'$ if it is $2$-regular or if it contains, as a minor, any of the excluded minors for $2$-regular matroids on fewer than $n$ elements.
\end{proof}

\Cref{2regtable} records the number of pairwise non-isomorphic $n$-element rank-$r$ matroids that are $2$-regular but not near-regular, for $n \le 15$.
Note that the two $10$-element $2$-regular matroids of rank-$3$ are the maximum-sized $2$-regular matroids known as $T_3^2$ and $S_{10}$ \cite{Semple1998}.

\begin{table}[ht]
  \begin{tabular}{r|r r r r r r r r r r r}
    \hline
$r \ba n$ & 5 & 6 & 7 &  8 &   9 &  10 &   11 &    12 &     13 &      14 &      15 \\
    \hline
        2 & 1 &   &   &    &     &     &      &       &        &         &         \\
        3 & 1 & 1 & 2 &  4 &   3 &   2 &      &       &        &         &         \\
        4 &   &   & 2 & 17 &  62 & 113 &  132 &    89 &     45 &      14 &       5 \\
        5 &   &   &   &  4 &  62 & 502 & 2156 &  5357 &   8337 &    8685 &    6338 \\
        6 &   &   &   &    &   3 & 113 & 2156 & 18593 &  88191 &  258318 &  511593 \\
        7 &   &   &   &    &     &   2 &  132 &  5357 &  88191 &  732667 & 3637691 \\
        8 &   &   &   &    &     &     &      &    89 &   8337 &  258318 & 3637691 \\
        9 &   &   &   &    &     &     &      &       &     45 &    8685 &  511593 \\
       10 &   &   &   &    &     &     &      &       &        &      14 &    6338 \\
       11 &   &   &   &    &     &     &      &       &        &         &       5 \\
       \hline                                                    
    Total & 2 & 1 & 4 & 25 & 130 & 732 & 4576 & 29486 & 193146 & 1266701 & 8311254 \\
       \hline
  \end{tabular}
  \caption{The number of $3$-connected $2$-regular $n$-element rank-$r$ matroids with a $\{U_{2,5},U_{3,5}\}$-minor, for $n \le 15$.}
  \label{2regtable}
\end{table}

We conjecture that there are no excluded minors for the class of $2$-regular matroids on more than 15 elements.
\begin{conjecture}
  \label{2regexminorconj}
  A matroid $M$ is $2$-regular if and only if $M$ has no minor isomorphic to
$U_{2,6}$, $U_{3,6}$, $U_{4,6}$, $P_6$,
$F_7$, $F_7^*$, $F_7^-$, $(F_7^-)^*$, $F_7^=$, $(F_7^=)^*$,
$\AG(2,3)\ba e$, $(\AG(2,3)\ba e)^*$, $(\AG(2,3)\ba e)^{\Delta Y}$, $P_8$, $P_8^-$, $P_8^=$, and $\TQ_8$.
\end{conjecture}

\begin{table}[htb]
  \begin{tabular}{ c c c c }
    \hline
    $M$ & $\mathbb{P}_M$ & $\max\{i : M \in \mathcal{M}(\mathbb{H}_i)\}$ & $\left|\Delta(M)\right|$ \\ 
    \hline
    $U_{2,6}$   & $\mathbb{U}_3$        & 6  & 3 \\
    $U_{3,6}$   & $\mathbb{P}_{U_{3,6}}$& 6  & 1 \\ 
    $F_7$       & $\GF(2)$              & -- & 2 \\
    $F_7^-$     & $\mathbb{D}$          & 2  & 2 \\
    $F_7^=$     & $\mathbb{K}_2$        & 2  & 2 \\ 
$\AG(2,3)\ba e$ & $\mathbb{S}$          & -- & 3 \\
    $P_8$       & $\mathbb{D}$          & 2  & 1 \\
    $P_8^-$     & $\mathbb{K}_2$        & 2  & 1 \\ 
    $P_8^=$     & $\mathbb{H}_4$        & 4  & 1 \\
    $\TQ_8$     & $\mathbb{K}_2$        & 2  & 1 \\ 
    \hline
  \end{tabular}
  \caption{The excluded minors for $2$-regular matroids on at most 15 elements, their universal partial fields, and how many inequivalent $\GF(5)$-representations they have. 
    We list one representative~$M$ of each $\Delta Y$-equivalence class $\Delta(M)$.%
  }
  \label{u2upfs}
\end{table}

We also calculated the universal partial fields for each excluded minor for the class of $2$-regular matroids, as shown in \cref{u2upfs}.  The only as-yet-undefined partial field is:
\begin{multline*}
  \mathbb{P}_{U_{3,6}} = (\mathbb{Q}(\alpha,\beta,\gamma,\delta), \langle-1, \alpha, \beta, \gamma, \delta, \alpha-1, \beta-1, \gamma-1, \delta-1, \\ \alpha-\beta, \gamma-\delta, \beta-\delta, \alpha-\gamma, \alpha\delta-\beta\gamma, \alpha\delta-\beta\gamma-\alpha+\beta+\gamma-\delta\rangle),
\end{multline*}
where $\alpha$, $\beta$, $\gamma$, and $\delta$ are indeterminates.
Note that there are no partial-field homomorphisms from $\mathbb{U}_3$ or $\mathbb{H}_4$ to $\GF(4)$, from $\mathbb{D}$ to fields of characteristic two, or from $\mathbb{S}$ to $\GF(5)$.
Thus, of the 17 matroids appearing in \cref{2regexminors} (and \cref{u2upfs}), all but $U_{3,6}$, $F_7^=$, $(F_7^=)^*$, $P_8^-$ and $\TQ_8$ are not representable over either $\GF(4)$ or $\GF(5)$.
On the other hand, 
we have the following:


\begin{lemma}
  \label{special4guys}
  The matroids $U_{3,6}$, $F_7^=$, $(F_7^=)^*$, $P_8^-$ and $\TQ_8$ are
  $\mathbb{K}_2$-representable, and
  representable over all fields of size at least four.
\end{lemma}
\begin{proof}
  It suffices to show that each of these matroids is $\mathbb{K}_2$-representable, and
  this follows directly from the universal partial fields calculations given in \cref{u2upfs}.

  Alternatively,
  observe that
$$\begin{bmatrix}
  1 & 1 & 1 \\
  1 & \alpha & \beta \\
  1 & \gamma & \delta \\
\end{bmatrix}$$
is a 
$\mathbb{P}_{U_{3,6}}$-representation of $U_{3,6}$, and 
let $\phi : \mathbb{P}_{U_{3,6}} \rightarrow \mathbb{K}_2$ be given by $\phi(\alpha) = -\alpha$, $\phi(\beta) = -1/\alpha$, $\phi(\gamma) = (\alpha-1)/\alpha$, $\phi(\delta) = 1-\alpha$.
It is easily verified that $\phi$ is a partial-field homomorphism.
%
It is also easy to check that the following are reduced
$\mathbb{K}_2$-representations for $F_7^=$, $\TQ_8$, and $P_8^-$, respectively (labelled as in \cref{fanosfig,tp8fig,p8fig}, where for $P_8^-$, we relax $\{e,f,g,h\}$).

\begin{multicols}{2}
$$\kbordermatrix{
  & d & e & f & g \\
a & 1 & 1 & 0 & 1 \\
b & 1 & 0 & 1 & 1 \\
c & 0 & 1 & \alpha & 1
}$$
$$\kbordermatrix{
  & 8 & 6 & 4 & 2\\
1 & 0 & \alpha & 1 & 1\\
7 & 1 & 0 & \alpha & \alpha - 1\\
5 & 1 & \alpha & 0 & \alpha \\
3 & 1 & \alpha - 1 & 1 & 0
}$$
\end{multicols}
$$\kbordermatrix{
  & d & e & g & h\\
a & 1 & 1 & 1 & \alpha+1\\
b & 1 & 0 & \alpha+1 & \alpha+1 \\
c & 1 & -\alpha & 1 & 0\\
f & 0 & 1 & 1 & 1
}$$
%
%
%
\end{proof}


\begin{corollary}
  Let $M$ be an excluded minor for the class of matroids representable over all fields of size at least four.
  Suppose that \cref{2regexminorconj} holds, or $|E(M)| \le 15$.
  Then, either
  \begin{enumerate}
    \item $M$ has a proper $\{U_{3,6}, F_7^=, (F_7^=)^*, P_8^-, \TQ_8\}$-minor, or
    \item $M$ is isomorphic to one of $U_{2,6}$, $U_{4,6}$, $P_6$, $F_7$, $F_7^*$, $F_7^-$, $(F_7^-)^*$, $\AG(2,3)\ba e$, $(\AG(2,3)\ba e)^*$, $(\AG(2, 3)\ba e)^{\Delta Y}$, $P_8$, and $P_8^=$.
  \end{enumerate}
\end{corollary}

Finally, we remark on the number of inequivalent $\GF(5)$-representations that the excluded minors for $2$-regular matroids possess.
As there is a partial-field homomorphism from $\mathbb{U}_3$ to $\mathbb{H}_5$~\cite{vanZwam2009}, and
$\phi : \mathbb{P}_{U_{3,6}} \rightarrow \mathbb{U}_3$ given by $\phi(\alpha) = \frac{\alpha-1}{\alpha}$, $\phi(\beta) = \frac{\gamma-1}{\gamma}$, $\phi(\gamma) = \frac{1-\alpha}{\beta-\alpha}$, and $\phi(\delta) = \frac{1-\gamma}{\beta-\gamma}$
is a partial-field homomorphism, the matroids $U_{2,6}$ and $U_{3,6}$ have precisely six inequivalent $\GF(5)$-representations.
For $\mathbb{P} \in \{\mathbb{D}, \mathbb{K}_2\}$, there is a partial-field homomorphism from $\mathbb{P}$ to $\mathbb{H}_2$ but none from $\mathbb{P}$ to $\mathbb{H}_3$~\cite{vanZwam2009}, so $F_7^-$, $F_7^=$, $P_8$, $P_8^-$, and $\TQ_8$ have precisely two inequivalent $\GF(5)$-representations.
As the universal partial field of $P_8^=$ is $\mathbb{H}_4$, the matroid $P_8^=$ has precisely four inequivalent $\GF(5)$-representations.

\bibliographystyle{abbrv}
\bibliography{lib}

\end{document}